\documentclass[a4paper]{article}
\pdfoutput=1
\usepackage{amsmath, amsrefs, amssymb, amsthm, xypic}
\usepackage[pdftex,
            pdftitle={Concrete Categories in Homotopy Type Theory},
            pdfauthor={James Cranch},
            pdfsubject={Homotopy Type Theory},
            pdfkeywords={homotopy,coherence,type theory,category theory},
            pdfproducer={LaTeX},
            pdfcreator={pdflatex}]{hyperref}

\title{Concrete Categories in Homotopy Type Theory}
\author{James Cranch}

\newcommand{\isom}{\cong}
\newcommand{\I}[1]{(\infty,#1)}
\newcommand{\tI}[1]{\texorpdfstring{$\I{#1}$}{(infty,#1)}}
\newcommand{\catname}[1]{\mathrm{#1}}
\newcommand{\Type}{\catname{Type}}
\newcommand{\ap}{\operatorname{ap}}
\newcommand{\transport}{\operatorname{transport}}
\newcommand{\refl}{\operatorname{refl}}
\newcommand{\arr}{\operatorname{arr}}
\newcommand{\obj}{\operatorname{obj}}
\newcommand{\id}{\operatorname{id}}
\newcommand{\inl}{\operatorname{inl}}
\newcommand{\inr}{\operatorname{inr}}

\newcommand{\ident}{\operatorname{ident}}
\newcommand{\cmp}{\operatorname{cmp}}
\newcommand{\conf}{\operatorname{conf}}
\newcommand{\Fun}{\operatorname{Fun}}
\newcommand{\Fin}{\operatorname{Fin}}
\newcommand{\Ord}{\operatorname{Ord}}
\newcommand{\bN}{\mathbb{N}}
\newcommand{\bNt}{\bN_{-2}}
\newcommand{\nil}{\operatorname{nil}}
\newcommand{\cons}{\operatorname{cons}}
\newcommand{\unit}{\operatorname{unit}}
\newcommand{\assoc}{\operatorname{assoc}}
\newcommand{\timeso}[1]{\mathop{\times}_{#1}}
\newcommand{\cC}{\mathcal{C}}
\renewcommand{\cD}{\mathcal{D}}
\newcommand{\isequiv}{\operatorname{is-equiv}}
\newcommand{\hfibre}{\operatorname{hfibre}}

\newdir{pb}{:(1,-1)@^{|-}}
\def\pb#1{\save[]+<12 pt,0 pt>:a(#1)\ar@{pb{}}[]\restore}

\theoremstyle{plain}
\newtheorem{thm}{Theorem}
\theoremstyle{remark}
\newtheorem{remark}[thm]{Remark}
\newtheorem{lemma}{Lemma}[thm]
\newtheorem{claim}{Claim}[thm]

\newenvironment{innerproof}[1]%
  {\begin{proof}[#1]
   }%
  {\end{proof}}

\begin{document}
\maketitle

\begin{abstract}
We introduce some classes of genuine higher categories in homotopy
type theory, defined as well-behaved subcategories of the category of
types. We give several examples, and some techniques for showing other
things are not examples. While only a small part of what is needed, it
is a natural construction, and may be instructive for people seeking
to provide a fully general construction.
\end{abstract}

\section{Introduction}

\subsection{Categories in homotopy type theory}

Homotopy type theory is a recently-developed foundational approach to
mathematics. The key idea is that a certain flavour of type theory can
be given a homotopical interpretation, in which a type is viewed as a
space, and a function is viewed as a continuous map. This
interpretation provides a rich semantics, and many internal
definitions can be made which harmoniously represent classical results
in homotopy theory as structural results about type theory. The basic
reference is the book\cite{HoTT-book}.

Inevitably, any foundational approach to mathematics will be judged in
some small part on its ability to comfortably represent category
theory, which has become an essential tool in organising modern
mathematics\cites{Borc,MacL}.

Thus far the author knows of one attempt, by Ahrens, Kapulkin and
Shulman\cite{AhKaSh}\cite{HoTT-book}*{Chapter 9}, to undertake
category theory in homotopy type theory. That attempts deals with
$1$-categories, rather than $\I1$-categories. In other words, the type
of homomorphisms between any two objects is homotopically
discrete. (Just as ordinary homotopy theory contains a theory of sets,
which can be represented as discrete spaces, homotopy type theory
contains a more classical type theory within, made up of those types
satisfying a similar kind of discreteness condition).

However, homotopy type theory studies types which are more general
than sets. Usually maps between structures built from such types
cannot be expected to be discrete.

Indeed, since homotopy type theory uses $\I0$-categories to model
types; it is natural to feel that $\I1$-categories are the most
appropriate concept of category in this setting, just as
Joyal\cite{Joy} and Lurie\cite{Lur} have provided ample evidence that
they are often tractable and useful in topology.

The purpose of this paper is to produce some genuine $\I1$-categories
in homotopy type theory. Our approach is certainly not fully general,
but our examples encompass a range of $\I1$-categories that one might
wish to work with. We also attempt to sketch some perceptions of the
limitations of our approach.

Also available is a library of code\cite{JDC-HoTT} written in the
dependently-typed programming language Agda\cite{Agda}, demonstrating
these concepts in practice; this is based on a homotopy type theory
library provided by Brunerie and coworkers\cite{Brun-HoTT}. At
appropriate points in what follows, we reference this library.

\subsection{Conventions}

This paper is written in an informal form of type theory, roughly as
used in the book \cite{HoTT-book}. We do not emphasise universes; the
reader who cares can identify appropriate universe levels for
themselves.

We write $x\equiv_Ay$ for the path type between two elements $x,y:A$
(we avoid using the phrase \emph{identity type}, saving the word
\emph{identity} for use in its categorical sense) and write simply
$x\equiv y$ if $A$ is obvious from context.

We use the dot $\cdot$ to denote composition of paths.

We call a $(-1)$-truncated type a \emph{proposition}; the book
\cite{HoTT-book} calls these \emph{mere propositions}, but we have no
use for any other meaning of the word and do not wish to sound
demeaning about them. Similarly, a $(-2)$-truncated type is called a
\emph{set}.

We use the phrase \emph{subcategory} in a very vague sense: we mean
the domain of a certain sort of functor. The functor in question is
always required to satisfy some kind of faithfulness condition (which
we will make clear as required), but never any kind of injectivity, or
essential injectivity, on objects. We feel that this is less uncommon
than it sounds: most practical uses of subcategories in mathematics
are similar.

\section{Inbuilt coherence: the category of types}

There are grave problems associated with naive attempts to define
categories, or higher categories, in homotopy type theory: there is an
infinite amount of data required, of a type which increases
progressively in complexity.

We may start with a set of objects $\obj:\Type$, and a dependent type
of morphisms $\hom:\obj\rightarrow\obj\rightarrow\Type$. This lacks
the basic structure of identities and composition, so we require
elements as follows:
$$\id:(x:\obj)\rightarrow\hom(x,x)$$
and
$$\cmp:(x,y,z:\obj)\rightarrow\hom(y,z)\rightarrow\hom(x,y)
  \rightarrow\hom(x,z).$$

This lacks the associative and unit laws of composition, so we must
add something (for example) whose content is that 
$$\cmp(\cmp(f,g),h) = \cmp(f,\cmp(g,h))$$
for all composable strings of morphisms.

However, in homotopy type theory, this does not assert that those two
are indistinguishable, merely that they are homotopic. As it happens,
we can form two different chains of composites of such homotopies
showing that
$$\cmp(\cmp(\cmp(f,g),h),k) = \cmp(f,\cmp(g,\cmp(h,k))).$$ We must
assert that these are equal, but this does not make contractible the
space of composites of five maps. Things continue becoming more
complex in this way. While appropriate structures in traditional
foundations are well-known\cite{Lein}, the problem of specifying the
resulting data in homotopy type theory is unsolved at the time of
writing.

The starting point of our work is the observation that, while nobody
has written down a general definition of $\I1$-categories in homotopy
type theory, there is one fully coherent example built in. That
category is the category $\Type$ of types, and functions between them.

In Agda, for example, we define the composition of two functions by
the usual formula
$$(g\circ f)(x) = g(f(x)).$$

But the result is that Agda normalises both $(h\circ g)\circ f$ and
$h\circ(g\circ f)$ to
$$\lambda x.h(g(f(x))),$$
and, as a result, the associativity of composition $(h\circ g)\circ
f\equiv h\circ(h\circ f)$ is a triviality.

The same goes for higher associativity laws. For example, consider the
``associativity pentagon'':
\begin{displaymath}
\xymatrix@C=0pt{
&(f\circ(g\circ h))\circ k\ar[rr]^{\equiv}&&f\circ((g\circ h)\circ k)\ar[dr]^{\equiv}&\\
((f\circ g)\circ h)\circ k\ar[ur]^{\equiv}\ar[ddrr]_{\equiv}&&&&f\circ(g\circ(h\circ k))\\
&&&&\\
&&(f\circ g)\circ(h\circ k)\ar[uurr]_{\equiv}&&}
\end{displaymath}
If we wish to verify that this commutes: that the two chains of
associativity identities connecting $((f\circ g)\circ h)\circ k$ and
$f\circ(g\circ(h\circ k))$ agree, then Agda can verify immediately
that both are simply reflexivity on
$$\lambda x.k(h(g(f(x)))),$$
and so are equal, by reflexivity of identity. This pattern continues:
all the structure of $\Type$ as a higher category is supplied in a
straightforward fashion by the underlying type theory.

\section{Inheriting coherence: $n$-concreteness}

As we have seen, the category $\Type$ has excellent properties within
homotopy type theory. Even so, it's only a single example of an
$(\infty,1)$-category, not a general approach to $(\infty,1)$-category
theory.

However, we can use it as a starting point for building more: for any
$n:\bNt$, we can define a notion of $\I1$-categories with a functor to
$\Type$, which is a $n$-truncated map on each homtype. We can call
these \emph{$(n+2)$-concrete $\I1$-categories}.

The coherence of the category $\Type$ automatically supplies the
desired coherence of our category in high degrees. However we must
explicitly supply the categorical structure in low degrees; this
structure is that of a fully weak $(n+1,1)$-category, so general
definitions require machinery and work which we are unwilling to
undertake here (see Leinster's book\cite{Lein} for a survey of
approaches).

The situation is perhaps best made clear by example; we talk through
the cases $n=0,1,2,3$ below: there are uncontroversial definitions of
categories and bicategories.

\subsection{$0$-concrete \tI1-categories}

The notion of $0$-concrete $\I1$-category is the notion of a full
subcategory of $\Type$.

We could specify such a thing simply by giving a type of objects $\obj$
and a realisation map $\obj^+:\obj\rightarrow\Type$, and define
$$\hom(x,y) = \left(\obj^+(x)\rightarrow\obj^+(y)\right).$$

More elaborately, one could give a type of objects and a realisation
map as above, and choose a map
$$\hom:\obj\rightarrow\obj\rightarrow\Type$$
together with, for each $x$ and $y$, an equivalence
$$\conf(x,y):\hom(x,y)\stackrel{\sim}{\rightarrow}\hom'(x,y),$$
the ``conformity map'' where $\hom'(x,y) = \left(\obj^+(x) \rightarrow
\obj^+(y)\right)$ as used before. Clearly these are the equivalent
concepts, and it is a matter of convenience which we choose to use; we
shall use the latter in what follows, since it is more similar to our
other definitions.

While we will give some useful examples below, our interest in the
notion of $0$-concrete $\I1$-categories is largely due to their status
as the simplest of the family of $n$-concrete $\I1$-categories.

Here is a result that gives some idea of the limitation of this
concept:
\begin{thm}
Consider the disjoint union $*\sqcup *$ of two copies of the terminal
category. This is not a $0$-concrete $\I1$-category.
\end{thm}
\begin{proof}
In fact, it's not a full subcategory of the category of spaces under
ordinary foundations. Suppose the two spaces representing the two
objects are $X$ and $Y$. Now, if $X$ has a point $x$, then there is a
constant map from $Y$ to $X$ with image $x$. However, if $X$ is empty,
then there is an inclusion map from $X$ to $Y$. Either way, there is
some map between them.

This argument does not work as stated in homotopy type theory, since
it uses the law of excluded middle to argue that $X$ must either be
empty or have a point. However, we can use double negation to recover
something similar: for types $X$ and $Y$, if $\neg(X\rightarrow Y)$,
we can show that $\neg(\neg X)$ and $\neg Y$.

We get
$$\neg(X\rightarrow Y)\rightarrow\neg Y$$
which of course means
$$((X\rightarrow Y)\rightarrow\bot)\rightarrow(Y\rightarrow\bot)$$
by composing with the constant map $Y\rightarrow (X\rightarrow Y)$.

And we get
$$\neg(X\rightarrow Y)\rightarrow\neg(\neg X)$$
from the inclusion $\bot\rightarrow Y$.

Hence, if we have two objects in a $0$-concrete $\I1$-category, then
we cannot simultaneously have $\neg(\obj^+(x)\rightarrow\obj^+(y))$
and $\neg(\obj^+(y)\rightarrow\obj^+(x))$: the former implies
$\neg\neg\obj^+(x)$ and the latter implies $\neg\obj^+(x)$, a
contradiction.
\end{proof}

\subsection{$1$-concrete \tI1-categories}

A $1$-concrete $\I1$-category is a subcategory of $\Type$ where, as
morphisms, we choose some connected components of the homtypes in
$\Type$. We cannot do this freely: we must choose the connected
components of the identity morphisms, and given any two morphisms we
have chosen, we must also choose the connected component of their
composite.

More formally, it consists of:
\begin{itemize}
\item A type $\obj$ of objects;
\item An object realisation map $\obj^+:\obj\rightarrow\Type$;
\item For every pair $x,y:\obj$, a type of homomorphisms $\hom(x,y)$;
\item For every pair $x,y:\obj$, a homomorphism realisation map
    $$\hom^+:\hom(x,y)\rightarrow\hom'(x,y),$$
  where $\hom'(x,y)=\left(\obj^+(x)\rightarrow\obj^+(y)\right)$;
\item For every pair $x,y:\obj$, an element $\conf(x,y)$ of the
  proposition that $\hom^+:\hom(x,y)\rightarrow\hom'(x,y)$ is
  $1$-truncated (the ``conformity'');
\item For every $x:\obj$, an element $\ident'(x)$ of the homotopy
  fibre of $\hom^+$ at the point $\id_{\obj^+(x)}$.
\item For every $x,y,z:\obj$, and $g:\hom(y,z)$ and $f:\hom(x,y)$, an
  element $\cmp'(g,f)$ of the homotopy fibre of $\hom^+$ at the point
  $\hom^+(g)\circ\hom^+(f)$.
\end{itemize}

This notion is already quite powerful, and using it we can comfortably
express many categories that we might choose to care about, as will be
seen in the next section.

\subsection{$2$-concrete \tI1-categories}

The next stage up, the $2$-concrete $\I1$-category, is a subcategory
of $\Type$ where we are allowed to choose a set of copies of each
morphism.

This requires still more data and axioms to be given by hand. We need
a choice of preimage of the identity maps, and of each composition,
much as before. But we now need to impose category axioms on this
structure: we need to impose the left and right unit axioms, and the
associativity axioms, to ensure that those choices of connected
components give genuine categories.

More formally, the structure consists of all the structure of a
$1$-concrete $\I1$-category, except that the conformity element
$\conf$ asserts that the maps $\hom^+$ are $0$-truncated, and elements
of the following types (for all $x,y,z,w:\obj$, $f:\hom(x,y)$,
$g:\hom(y,z)$ and $h:\hom(z,w)$ as appropriate):
\begin{align*}
\unit^l &: \cmp(\ident(y),f) \equiv f \\
\unit^r &: \cmp(f,\ident(x)) \equiv f \\
\assoc  &: \cmp(\cmp(h,g),f) \equiv \cmp(h,\cmp(g,f)).
\end{align*}
Here we define $\ident$ and $\cmp$ to be the first component of
  $\ident'$ and $\cmp'$ respectively, so that they have types
\begin{align*}
\ident &: (x : \obj) \rightarrow \hom(x,x)\\
\cmp   &: \hom(y,z) \rightarrow \hom(x,y) \rightarrow \hom(x,z).
\end{align*}

\subsection{$3$-concrete \tI1-categories and beyond}

By now, hopefully the pattern is becoming clear. A $3$-concrete
$\I1$-category will have a conformity type that is weaker still: it
only asserts that the maps $\hom^+$ are $1$-truncated.

This means that more structure should be supplied by hand: the
pentagon and triangle identities, familiar from the definition of a
bicategory (or a monoidal category) as in \cite{Borc}, need to be
imposed to ensure coherence of the unit and associativity laws.

In general, each time we increase the concreteness level, we need to
add more axioms simulating a weak $n$-category.

\section{Examples}

\subsection{\tI1-categories of types, sets, $n$-groupoids, etc}

The obvious examples of $0$-concrete $\I1$-categories simply consist of full
subcategories of the category of types on special sorts of types.

The trivial case is, of course, the $0$-concrete $\I1$-category of
types itself.

We could take as objects, instead, the $n$-truncated types for any
$n:\bNt$. For $n = 0, 1, \ldots$ we get the $\I1$-category of sets, or
of $1$-groupoids, and so on.

Another family of examples is what we get from using a singleton as
set of objects: this is a coherent version of the endomorphism monoid
of a type $X$, regarded as a $1$-object category.

We can produce the category of finite sets: there is a standard model
for nonempty finite ordered sets: we define $\Fin(n)$ for $n:\bN$ by
the constructors:
\begin{align*}
    0 &: \Fin(n+1)\\
    S &: \Fin(n) \rightarrow \Fin(n+1).
\end{align*}
This gives us a $0$-concrete $\I1$-category with $\obj=\bN$ and
$\obj^+ = \Fin$.

\subsection{The simplicial category $\Delta$}

The simplex category $\Delta$, the category of nonempty finite ordered
sets and order-preserving maps, fits into this scheme.

We also provide a convenient model $\Ord(0,0)$ for the ordered maps
from $\Fin(m)$ to $\Fin(n)$, with three constructors:
\begin{align*}
    0 &: \Ord(0,0)\\
    S^l &: \Ord(m,n+1)\rightarrow\Ord(m+1,n+1)\\
    S^r &: \Ord(m,n)\rightarrow\Ord(m,n+1).
\end{align*}
The semantics of $0$ are obvious; those of $S^l$ are defined by
\begin{align*}
    S^l(f)(0) &= 0,\\
    S^l(f)(i+1) &= f(i);
\end{align*}
and those of $S^r$ are defined by
$$S^r(f)(i) = f(i)+1;$$
This recursively defines a map
$$\Ord^+(m,n):\Ord(m,n)\longrightarrow(\Fin(m)\rightarrow\Fin(n))$$
for every $m$ and $n$.

It is straightforward to recursively define identities and
compositions for the type family $\Ord$, and also to show that these
coincide with the genuine identities and compositions under $\Ord^+$.

Moreover, standard methods permit one to show that $\Fin(n)$ and
$\Ord(m,n)$ are both sets.

\begin{thm}
The category $\Delta$ is a $1$-concrete $\I1$-category.
\end{thm}
\begin{proof}
We use object set $\obj=\bN$, and the object realisation map
$\Fin\circ S$ (the suspension is so that we get only nonempty finite
ordered sets).

Then we use $\Ord$ to define $\hom$, and then $\hom^+$ is the
recursively-defined map $\Ord^+$ defined above.

Identities and composition have already been discussed. All that
remains is the conformity. We find it helpful to prove the following:
\begin{claim}
An injection into a set has propositions as homotopy fibres.
\end{claim}
\begin{innerproof}{Proof of Claim}
It is easy to show that any two elements of the homotopy fibre are
equal.
\end{innerproof}
We put this claim to work on the map $\Ord^+(m,n)$. The codomain is
the type of functions $\Fin(m)\rightarrow\Fin(n)$, which is a set
since $\Fin(n)$ is one. It is not hard to prove that $0\neq S(i)$ for
all $i:\Fin(n)$, and thence to show recursively that the map $\Ord^+$
is injective.
\end{proof}

\subsection{Ahrens-Kapulkin-Shulman 1-categories}

The work\cite{AhKaSh} of Ahrens, Kapulkin and Shulman provides
examples of our theory. We refer to the notion of category they
consider as \emph{AKS-categories}.

Generalising the preceding example somewhat, examples of their theory
give examples of our theory:
\begin{thm}
Any AKS-category yields a $2$-concrete $\I1$-category.
\end{thm}
\begin{proof}
Suppose we have such a category, with object type $\obj$ and morphism
types $\hom(x,y)$ for $x,y:\obj$.

We inherit the type of objects as is. The object realisation map
$\obj^+$ takes an object $x$ to the type of pairs consisting of an
object $y$ and an element $f:\hom(y,x)$.

Note that $\obj^+(x)$ is a $1$-truncated type. That is because it is a
$\Sigma$-type; $\obj$ is $1$-truncated (this is \cite{AhKaSh}*{Lemma
  3.8}) and homsets in an AKS-category are genuine sets: they're
$0$-truncated and hence $1$-truncated. Hence, also, for all $x$ and
$y$ the type of maps from $\obj^+(x)$ to $\obj^+(y)$ is $1$-truncated.

We define the realisation $\hom^+$ as follows:
$$\hom^+(f)(z,g) = (z,f\circ g).$$

It is straightforward to define units and composition using this
definition; the maps require the left unit and associativity axioms
respectively.

The problem that remains is conformity. Given a map $f$ from
$\obj^+(x)$ to $\obj^+(y)$, we must show that the homotopy fibre of
$f$ under $\hom^+$ is a set. A $\Sigma$-type is $n$-truncated if the
base and all fibres are $n$-truncated. In this case the base is a set
because one axiom of an AKS-category is that homomorphisms form sets,
and the fibre is a set because it's a path type of the $1$-truncated
type $\obj^+(x)\rightarrow\obj^+(y)$.
\end{proof}

\subsection{Types as \tI0-categories}

Given a type $X$, it is natural to wish to regard $X$ as an
$\infty$-groupoid, which is an $\I0$-category: a degenerate case of an
$\I1$-category where all morphisms (given by path types) are
equivalences.

We can do this:
\begin{thm}
A type $X$ can be given the structure of a $1$-concrete $\I1$-category.
\end{thm}
\begin{proof}
We naturally take $\obj = X$, and we choose $\obj^+x$ to be the type
of paths to $x$. As promised, we also choose $\hom(x,y) =
(x\equiv_Xy)$. The proper definition of $\hom^+$ is very much like
that used in the subsection above:
$$\hom^+(e)(z,\epsilon) = (z,\epsilon\cdot e).$$
The identity and composites are quick checks, and conformity is also
rapidly proved by path induction.
\end{proof}

\subsection{Automorphism groups and categories of $n$-truncated maps}

Given a map $f:X\rightarrow Y$ between two types, there is a
proposition expressing that $f$ is $n$-truncated (for any
$n:\bNt$). It is also true that identity maps are $n$-truncated (for
all $n$), and composites of $n$-truncated maps are $n$-truncated.

That gives that there are $1$-concrete $\I1$-categories of types and
$n$-truncated maps, or of any given type of types and $n$-truncated
maps between them.

One special case is when we take the object type to be a singleton and
$n=-2$: the resulting category is the one-object category of
self-equivalences of some given type: this $\I1$-category can be
regarded as the automorphism group of that type.

\subsection{Free categories}

One might reasonably wish to discuss the free category (on a specified
type of objects and specified types of morphisms between them).

So suppose given a type $\obj:\Type$ and a family
$\arr:\obj\rightarrow\obj\rightarrow\Type$.

We can define the homomorphisms inductively, as linked lists of
composable arrows:
\begin{align*}
 \nil &: \hom(x,x)\\
\cons &: \hom(y,z)\rightarrow\arr(x,y)\rightarrow\hom(x,z).
\end{align*}
As is normal for linked lists, there is a unital and associative
composition operation, which we denote by $\bullet$.

This structure fits into our framework:
\begin{thm}
A free category is a $1$-concrete $\I1$-category.
\end{thm}
\begin{proof}
We have $\obj$ and $\hom$ already. We define $\obj^+(x)$ to be the
type of ``homs to $x$'', in other words the $\Sigma$-type of pairs
consisting of an element $y:\obj$ and an element $f:\hom(y,x)$.

Then $\hom^+$ is defined by composition:
$$\hom^+(f)(x,g) = (x,g\bullet f).$$
Using this definition, $\ident'$ and $\cmp'$ are clear from the
algebraic properties of the composition.

What's left is the conformity. Standard methods, as in
\cite{HoTT-book}*{Section 2.12}, will prove that for any $x,y:\obj$
and for any $f,g:\hom(x,y)$, the type $\hom^+(f)\equiv\hom^+(g)$ is
equivalent to the type $f\equiv g$. This enables us to show that
$\hom^+$ is 1-truncated, via the following lemma:
\begin{lemma}
Suppose $k$ is any function such that, for any $x$, $y$, the map
$$\ap(k):(x\equiv y)\rightarrow (k(x)\equiv k(y))$$
is an equivalence. Then the homotopy fibre of $k$ at any point is a
proposition.
\end{lemma}
\begin{innerproof}{Proof of lemma}
We show that any two elements of the homotopy fibre are equal. This
can be simplified by using path induction to simplify one (but not
both) of the second components of the homotopy fibre to $\refl$.

So we aim to show that an element $(a,u)$ is equal to an element
$(b,\refl)$. Writing $e$ for the hypothesis that $\ap(k)$ is an
equivalence, we can show that $a\equiv b$ immediately from $e$ (as
$\pi_1(\pi_1(e(u)))$); we are left with the check
\begin{align*}
 &\transport(\lambda x\rightarrow f(x)\equiv f(b),\pi_1(\pi_1(e(u))),u)\\
=&\transport(\lambda x\rightarrow x\equiv f(b),\ap(f)(\pi_1(\pi_1(e(u)))),u)\\
=&!\ap(f)(\pi_1(\pi_1(e(u)))) \cdot u\\
=&!u \cdot u\\
=&\refl.\qedhere
\end{align*}
\end{innerproof}
This completes the proof.
\end{proof}

\section{Pointed types: a cautionary tale}
\label{ss:pointed-types}

Recall that the type of \emph{pointed types} is defined by
$$\Type^* = \Sigma(\Type,\id),$$
so that a pointed type $(X,x)$ consists of a type $X$ and an element
$x:X$ (the \emph{basepoint}).

Given two pointed types $(X,x)$ and $(Y,y)$, the type of \emph{pointed
  maps} between them is defined by
$$(X,x)\stackrel{*}{\rightarrow}(Y,y) = \Sigma(X\rightarrow Y, \lambda
f\rightarrow f(x)\equiv y).$$

This concept is ubiquitous in algebraic topology, and so one would
naturally want to form the category of pointed types and pointed maps
accordingly.

Unfortunately, the obvious approach doesn't work. This would be to
model it as a concrete $\I1$-category by forgetting the basepoint, so
taking $\obj^+(X,x) = X$ and $\hom^+(f,p) = f$. But then, consider
what happens when we take $X$ to be the singleton type $1$, and $*$ to
be its unique element, then the type of pointed maps
$$(1,*)\stackrel{*}{\rightarrow}(Y,y)$$
is contractible (since it is equivalent to the type of paths to $y$ in
$Y$), but the type $\hom'((1,*),(Y,y))$ is the type $1\rightarrow Y$,
which is equivalent to $Y$. The map $\hom^+$ is, under these
equivalences, the inclusion of $y$ into $Y$.

The problem is that this inclusion has homotopy fibre $y\equiv_Yy$,
the loop space of $Y$ at $y$, and this can only be expected to be
$n$-truncated if $Y$ is $(n+1)$-truncated (more precisely, if the
basepoint component of $Y$ is). Thus we cannot form an $n$-concrete
category of all pointed types by this process for any $n$.

One could wonder whether this was just an unfortunate choice of
$\obj^+$ and $\hom^+$, but one cannot do any better:
\begin{thm}
The category of pointed types is not an $n$-concrete category for any $n$.
\end{thm}
\begin{proof}
Suppose that the category of pointed types can be described as an
$n$-concrete category.

It is not possible for $\obj^+(X,x)$ to be $k$-truncated for every
pointed type $(X,x)$. Indeed, if it were, then $\hom'((X,x),(Y,y))$
would be $k$-truncated for every pair of pointed types, and then, as
$\hom^+((X,x),(Y,y))$ is an $n$-truncated map, $\hom((X,x),(Y,y))$
would be $(n+k)$-truncated, which is certainly not true for all pairs
of pointed spaces!

But then, the argument of the special case above goes through: the map
$\hom^+((1,*),(Y,Y))$ goes from a contractible type to a type which is
not in general $k$-truncated for any $k$, and hence cannot be an
$n$-truncated map in general.
\end{proof}

Plainly enough, the same difficulties may be expected to apply to most
other categories of structured types.

There are certain compromises that can be made. For example, the
category of pointed $n$-truncated types will certainly form an
$(n+1)$-concrete category. That is certainly less than the homotopy
theorist would wish for, but may be of considerable utility to the
algebraist.

Another trick is to truncate the defining equation of a pointed map,
defining instead
$$\hom((X,x),(Y,y)) = \Sigma((X\rightarrow Y), \lambda f\rightarrow
\tau_i(f(x)\equiv y)),$$
where $\tau_i$ is the $i$-truncation operator. In the case $i=-1$,
this may be interpreted as describing the category of pointed types
and maps which preserve only the connected component of the basepoint.

It is hard to escape the conclusion that this highlights an essential
deficiency of type theory in dealing with pointed types: it is hard to
see any way of dealing with them without explicitly having to handle
coherence at all levels.

\begin{remark}
\label{rmk:arrow-cat}
This analysis shows that we also can't hope to form the arrow category
of $\Type$ as a concrete category: the category whose morphisms are
pairs of types $X$, $Y$ equipped with a map $X\rightarrow Y$. Indeed,
were this to be possible, we could obtain a category of pointed types
by imposing the restriction that $X$ be contractible (which is a
proposition).
\end{remark}

\section{Spans of types: an open problem}

The \emph{category of spans}, and its variants, can be expected to be
of some importance; some of their uses in homotopy theory are
described in the author's PhD thesis\cite{JDC-PhD}.

We aim to describe a category with $\obj=\Type$, and where $\hom(X,Y)$
is the $\Sigma$-type of \emph{spans}: pairs consisting of a type $U$
and morphisms $f:U\rightarrow X$ and $g:U\rightarrow Y$. We write such
things as $(f;U;g)$; and draw them where possible as roof-shaped
diagrams:
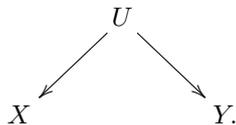
\begin{displaymath}
\xymatrix{&U\ar[dl]\ar[dr]&\\
X&&Y.}
\end{displaymath}
The identity span on a type $X$ consists entirely of identities on
$X$: it is $(\id_X;X;\id_X)$. Composition is defined by pullbacks: the
composite of spans \hbox{$X\leftarrow U\rightarrow Y$} and
\hbox{$Y\leftarrow V\rightarrow Z$} is
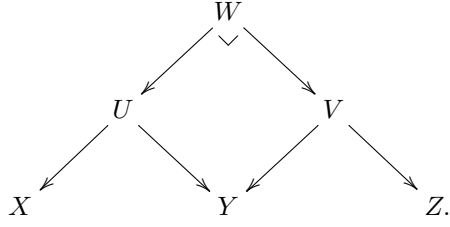
\begin{displaymath}
\xymatrix{&&W\pb{270}\ar[dl]\ar[dr]&&\\
&U\ar[dl]\ar[dr]&&V\ar[dl]\ar[dr]&\\
X&&Y&&Z.}
\end{displaymath}
In our symbolic notation, the composite of a span $(f;U;g)$ from $X$
to $Y$ with a span $(h;V;k)$ from $Y$ to $Z$ is
$$(h;V;k)\circ(f;U;g) = (f\pi_1;U\timeso{Y}V;k\pi_2).$$

This definition inspires an appropriate choice of $\obj^+$: we might
take $\obj^+X$ to be the type $\Type_{/X}$ of types over $X$: that is,
types equipped with a map to $X$. We could then define $\hom^+$ by a
``pull-push'' construction:
$$\hom^+(f;U;g)(A) = g_*f^*(A).$$

Here $f^* : \Type_{/X}\rightarrow\Type_{/U}$ denotes the pullback
along $f$: it replaces a set $\alpha:A\rightarrow X$ over $X$ with the
set $f^*\alpha : Z\timeso{X}U$ over $U$. Also, $g_* :
\Type_{/U}\rightarrow\Type_{/Y}$ denotes the pushforward along $g$: it
replaces a set $\beta:B\rightarrow U$ with the set
$g\beta:B\rightarrow Y$.

Naturally one could restrict various parts of the structure: for
example, restricting the objects only to certain families of types;
restricting the central objects in the spans, or restricting the class
of morphisms which are permitted in the spans. We might call any such
structure \emph{a category of spans}, but in the discussion below we
will continue to assume there are no such restrictions for simplicity.

It is of course reasonable to ask whether this structure, as
described, does indeed produce any $n$-concrete $\I1$-categories of
spans (for any $n$).

We can give a sense of the nature of this question by looking between
two simple examples of pairs of objects.

Firstly, we consider morphisms from $\emptyset$ to $1$. In this case
the type of spans is contractible: given a diagram $\emptyset
\leftarrow U\rightarrow 1$, the $U$ must be empty (and there is a
contractible type of empty types) and the maps are then chosen from a
contractible type of possibilities.

The type $\Type_{/\emptyset}$ is contractible, and the type
$\Type_{/1}$ is equivalent to $\Type$. Thus the type
$\hom'(\emptyset,1)$ is $(1\rightarrow\Type)\isom\Type$.

The map $\hom^+$ is, under these equivalences, the map
$1\rightarrow\Type$ picking out the empty type $\empty$. This map can
certainly be seen to be $(-1)$-truncated: emptiness is a proposition.

Secondly, however, we consider morphisms from $1$ to $1$. In this case
the type of spans is equivalent to $\Type$: all we do is freely choose
the intervening object $U$ in a diagram $1\leftarrow U\rightarrow 1$,
and then we have a contractible type of choices for the maps.

The map $\hom^+$ is then the map
$$\Type\longrightarrow(\Type\rightarrow\Type)$$
sending $U$ to the map $(U\times-)$.

Now, suppose we investigate what happens if we attempt to prove that
this map is $(-1)$-truncated. Suppose we have a map
$F:\Type\rightarrow\Type$; is its homotopy fibre $n$-truncated?

That is, given two pairs $(U,\alpha)$ and $(V,\beta)$, where
$U,V:\Type$, $\alpha : F \equiv (U\times-)$ and $\beta : F \equiv
(V\times-)$, what can we say about the type $(U,\alpha)
\equiv_{\hfibre(F)} (V,\beta)$?

To start with, we have
$$U\equiv U\times 1\equiv F(1)\equiv V\times 1\equiv V,$$
using $\alpha$ and $\beta$ respectively.

So, using path induction, we may as well suppose that $U=V$ and simply
ask about $\alpha \equiv_{F\equiv(U\times-)} \beta$; or, better yet,
discuss the type $(U\times-) \equiv (U\times-)$, which contains the
element $\alpha^{-1}\cdot\beta$.

This is not going to be $n$-truncated in general for any $n$: if $U$
has interesting self-equivalences, they will extend to $(U\times
-)$. For example, if $U$ is the boolean type $1\sqcup 1$, then
$\alpha^{-1}\cdot\beta$ may well exchange the summands.

However, even if $U$ is contractible, it is not clear what we can say:
while the author does not believe it is possible to write down any
element of the type
$$(V:\Type)\rightarrow V\equiv_{\Type} V$$
except $\lambda V\rightarrow\refl$, he has been unable to show that
this type is contractible, and hence it is unclear, at least with the
standard axioms of homotopy type theory, how to show that any category
of spans is $n$-concrete.

\section{Constructions}

In this section we list a few general constructions on concrete
$\I1$-categories.

\subsection{Increasing the concreteness level}

As one might expect, an $n$-concrete $\I1$-category can be viewed as
an $(n+1)$-concrete $\I1$-category. In general, the conformity axiom
for an $n$-concrete $\I1$-category trivially implies the conformity
axiom for an $(n+1)$-concrete $\I1$-category, and it also provides the
extra structure in degree $n$.

\subsection{Disjoint unions}

For any $n\geq 1$, the disjoint union $\cC\sqcup\cD$ of two
$n$-concrete $\I1$-categories $\cC$ and $\cD$ has the structure of an
$n$-concrete $\I1$-category.

Indeed, we take $\obj = \obj_\cC\sqcup\obj_\cD$, and we take $\obj^+$
to be defined as $\obj^+_\cC$ on $\obj_\cC$ and as $\obj^+_\cD$ on
$\obj_\cD$.

All the other structure is defined as it is in $\cC$ or $\cD$ as
appropriate (there is nothing to define whenever objects from both
$\cC$ and $\cD$ are involved).

As a result, the category $*\sqcup*$ \emph{is} a $1$-concrete
$\I1$-category.

\subsection{Products}

Products of $n$-concrete $\I1$-categories are $n$-concrete
$\I1$-categories, for $n\geq 1$. We demonstrate this explicitly for
$n=1$:
\begin{thm}
Let $\cC$ and $\cD$ be $1$-concrete $\I1$-categories. Then
$\cC\times\cD$ is also a $1$-concrete $\I1$-category.
\end{thm}
\begin{proof}
We take $\obj = \obj_\cC\times\obj_\cD$, and $\obj^+(x,y) =
\obj^+_\cC(x)\sqcup\obj^+_\cD(y)$.

Naturally, we define $\hom((x_1,y_1),(x_2,y_2)) =
\hom_\cC(x_1,x_2)\times\hom_\cD(y_1,y_2)$.

There is then an obvious candidate for the map $\hom^+$, which has
type
$$\hom_\cC(x_1,x_2)\times\hom_\cD(y_1,y_2) \longrightarrow
\obj^+_\cC(x_1)\sqcup\obj^+_\cD(y_1) \longrightarrow
\obj^+_\cC(x_2)\sqcup\obj^+_\cD(y_2),$$
namely to define $\hom^+(f,g) = \hom^+_\cC(f)\sqcup\hom^+_\cD(g).$

In this setup, the existence of suitable $\cmp'$ and $\ident'$ is
easy; the big problem is the conformity. This follows from the fact
that $\hom^+(f,g)$ is the composite of two maps; firstly a map which could reasonably be called $\sqcup$, from
$$(\obj^+_\cC(x_1)\rightarrow \obj^+_\cC(x_2)) \times
(\obj^+_\cD(y_1)\rightarrow \obj^+_\cD(y_2))$$
to
$$\left((\obj^+_\cC(x_1)\sqcup\obj^+_\cD(y_1)) \rightarrow
(\obj^+_\cC(x_2)\sqcup\obj^+_\cD(y_2))\right)$$
and $\hom^+_\cC(f)\times\hom^+_\cD(g)$. It is not a difficult exercise
to show that both these maps are $(-1)$-truncated.
\end{proof}

\section{Equivalences and univalence}
\label{s:univ}

It is a normal demand of category theory to be able to define
equivalences. It is particularly important in this setting: Ahrens,
Kapulkin and Shulman\cite{AhKaSh} discuss the utility of imposing a
\emph{univalence axiom}, which states that, between any two objects
$x$ and $y$, the natural map from the type of paths between $x$ and
$y$ to the type of equivalences between them is an equivalence.

We proceed to define equivalences in $n$-concrete $\I1$-categories in
the manner that one might expect: we use the machinery of equivalences
of types, together with some extra data to check that the given
structure in degrees up to $n$ agrees with that machinery.

Accordingly, we assume given a type $\isequiv(f)$ dependent upon types
$X$ and $Y$ and a function $f:X\rightarrow Y$, which expresses that
$f$ is an equivalence and which is a proposition for all $f$. Several
models are described in \cite{HoTT-book}*{Theorems 4.2.13, 4.3.2,
  4.4.4}. Using this we will define types $\isequiv_n(f)$ for $f$ a
morphism $f:\hom(x,y)$ in an $n$-concrete $\I1$-category, for
$n=0,1,2$.

\subsection{The $0$-concrete case}

The type $\isequiv_0(f)$ is defined simply to be
$\isequiv(\hom^+(f))$. This, of course, is a proposition.

\subsection{The $1$-concrete case}

We define the type $\isequiv_1(f)$ to be the type of pairs consisting
of:
\begin{itemize}
\item An element of $\isequiv(\hom^+(f))$; and
\item An element of the homotopy fibre of $\hom^+$ over the inverse
  $\hom^+(f)^{-1}$ thus described.
\end{itemize}
In other words, a map $f$ in a subcategory $\cC$ of $\Type$ is
invertible if it's invertible in $\Type$, and its inverse is also
contained in $\cC$.

This, again, is a proposition: it's a $\Sigma$-type whose base and
fibre are both propositions.

If we wish to choose the model for $\isequiv$ consisting of
bi-invertible morphisms, we can simplify this description: it consists
of morphisms $g, g':\hom(y,x)$ such that $\cmp(g,f)\equiv\id(x)$ and
$\cmp(f,g')\equiv\id(y)$.

\subsection{The $2$-concrete case}

We define the type $\isequiv_2(f)$ to be the type whose elements
consist of:
\begin{itemize}
\item An element of $\isequiv(\hom^+(f))$;
\item An element $(g,e)$ of the homotopy fibre of $\hom^+$ over the
  inverse $\hom^+(f)^{-1}$ thus described;
\item Paths $\cmp(f,g)\equiv\id$ and $\cmp(g,f)\equiv\id$.
\end{itemize}

Again, this is a proposition: it's fibred over the proposition
$\isequiv(\hom^+(f))$, and the standard proof that any two inverses
are equal proves that any appropriate elements of the homotopy fibre
of $\hom^+$ are equal.

As before, this admits a simplication if we use bi-invertibility as
our definition of equivalence: again we just need $g,g':\hom(y,x)$
with $\cmp(g,f)\equiv\id(x)$ and $\cmp(f,g')\equiv\id(y)$.

\section{Functors}

If one is serious about doing category theory with $n$-concrete
$\I1$-categories, then one must certainly wish to define functors
between them. A direct definition, sending objects to objects, and
morphisms to morphisms, and so on, appears to have all the
deficiencies that a direct definition of categories would have: the
need for an infinite sequence of coherence data.

However, there is a standard trick for representing functors using
only a well-developed theory of categories, using the notion of a
\emph{cocartesian fibration}. This approach is developed in the
Joyal-Lurie theory of $\I1$-categories in \cite{Lur}*{Section 2.4 and
  thereafter}.

Suppose, therefore, we have an $n$-concrete $\I1$-category on object
set $A\sqcup B$, and no homomorphisms from anything in $B$ to anything
in $A$:
$$(a:A)(b:B)\rightarrow\neg\hom(\inr(b),\inl(a)).$$

Given that, this category can be regarded as being over the category
with two objects and one non-identity arrow. We call it an
\emph{arrowlike category}.

A morphism $f:\hom(\inl(a),\inr(b))$ is \emph{cocartesian} if, for all
$b':B$, the map $\lambda g\rightarrow\cmp(g,f)$ induces an equivalence
$$\hom(\inr(b),\inr(b'))\longrightarrow\hom(\inl(a),\inr(b')).$$

We say that an arrowlike category as described above is a
\emph{cocartesian fibration} if every object in $A$ has a cocartesian
morphism out of it.

To start with, the concept of a cocartesian morphism is well-behaved:
\begin{thm}
The type of proofs that a morphism is cocartesian is a proposition.
\end{thm}
\begin{proof}
It's a dependent function type, valued in types of equivalences, all
of which are propositions.
\end{proof}

In fact, more than this is true, providing we use a univalence axiom
(as discussed in Section \ref{s:univ} above):
\begin{thm}
In a univalent $n$-concrete $\I1$-category, the type of cocartesian
morphisms out of any object is a proposition.
\end{thm}
\begin{proof}
In fact (following the pattern so far) we prove this in detail only
for $n=0,1,2$.

First we show that, given any two cocartesian morphisms $f:\hom(x,y)$
and $g:\hom(x,z)$, there is an equivalence between $y$ and $z$.

The cocartesian nature of $f$ gives an element $i:\hom(y,z)$ such that
$\cmp(i,f)\equiv g$. Similarly, there is an element $j:\hom(z,y)$ such
that $\cmp(j,g)\equiv f$.

Now,
$$\cmp(\cmp(i,j),g)\equiv \cmp(i,\cmp(j,g))\equiv \cmp(i,f) \equiv
g,$$ but since $\cmp(-,g)$ is an equivalence this means that
$\cmp(i,j)\equiv\id(z)$.

Similarly $\cmp(\cmp(j,i),f)\equiv f$ and hence
$\cmp(j,i)\equiv\id(y)$.

By the discussion in Section \ref{s:univ}, this gives us our
equivalence. In general, for an $n$-concrete $\I1$-category for $n>2$,
we should have to work more.

Now the appropriate univalence axiom gives that $y\equiv z$ and
$\hom^+(i)$ maps to $\id(\obj^+(y))$, and hence that $f\equiv g$.
\end{proof}

As a corollary, we get that the notion of cocartesian fibration is
well-behaved:
\begin{thm}
The type of proofs that an arrowlike category is a cocartesian
fibration is a proposition.
\end{thm}
\begin{proof}
This type is an dependent function type, and by the previous theorem
it is valued in propositions, and hence a proposition itself.
\end{proof}

\begin{remark}
While we can \emph{define} functors, we have no chance of forming
concrete functor categories. Indeed, we can't even form the arrow category
$\Fun(\Delta^1,\Type)$, as mentioned above in Remark \ref{rmk:arrow-cat}.
\end{remark}

\section{Prospects for further work}

\subsection{Further constructions}

Clearly the methods described above do not constitute a full
development of category theory. One may reasonably ask about
$n$-concrete versions of other popular constructions: with what
truncation hypotheses can they be defined?

\subsection{Uniform definitions}

The definitions above depend on an understanding of notions of
$(n,1)$-category; we have restricted detailed discussion to cases of
small $n$ where appropriate definitions are well-known.

Nevertheless, families of general definitions exist\cite{Lein}, and it
would perhaps be worthwhile to see if any of them can painlessly be
implemented in homotopy type theory.

The aim would be a well-defined family of definitions of $n$-concrete
$\I1$-category, valid for all $n:\bN$.

\subsection{Concrete categories in exotic homotopy type theories}

At present only one homotopy type theory has received extensive study:
the one modelled by the homotopy theory of spaces.

It was once the case that the only homotopy theory that was studied
was the homotopy theory of spaces. However, with the help of the
language of model categories \cite{ModCats}, it was progressively
realised that this is just one in a vast family of homotopy theories,
many of them helpful even in furthering understanding of spaces
themselves. A complete list of examples would be longer than this
paper; one very modest example is the theory of pointed spaces, where
preservation of the basepoint is forced.

The author suspects that exotic homotopy type theories, corresponding
to other homotopy theories, will soon receive heavy attention. This
will naturally augment the collection of concrete categories: given a
type theory containing universes of pointed types, the difficulties of
subsection \ref{ss:pointed-types} would vanish altogether.

\end{document}